\documentclass[10pt]{amsart}%
\usepackage{graphicx,amscd,color,amsmath,amsfonts,amssymb,geometry}
\usepackage{amsmath}
\usepackage{amsfonts}
\usepackage{amssymb}
\usepackage{graphicx}%
\providecommand{\U}[1]{\protect\rule{.1in}{.1in}}
\newtheorem{theorem}{Theorem}[section]

\begin{document}
\title[A note on the polynomial Bohnenblust--Hille inequality]{A note on the polynomial Bohnenblust--Hille inequality}
\author[D. Nu\~{n}ez-Alarc\'{o}n]{D. Nu\~{n}ez-Alarc\'{o}n\textsuperscript{*}}
\address{Departamento de Matem\'{a}tica,\newline\indent Universidade Federal da Para\'{\i}ba,\newline\indent 58.051-900 - Jo\~{a}o Pessoa, Brazil.}
\email{danielnunezal@gmail.com}
\thanks{\textsuperscript{*}Supported by Capes.}

\begin{abstract}
Recently, in paper published in the \textit{Annals of Mathematics}, it was
shown that the Bohnenblust--Hille inequality for (complex) homogeneous
polynomials is hypercontractive. However, and to the best of our knowledge,
there is no result providing (nontrivial) lower bounds for the optimal
constants for $n$-homogeneous polynomials ($n>2$). In this short note we
provide lower bounds for these famous constants.

\end{abstract}
\maketitle

\section{Introduction}

The Bohnenblust--Hille inequality for complex homogeneous polynomials
(\cite{bh}, 1931) asserts that there is a function $D:\mathbb{N}
\rightarrow\lbrack1,\infty)$ such that for every $m$-homogeneous polynomial
$P$ on $\mathbb{C}^{N}$, the $\ell_{\frac{2m}{m+1}}$-norm of the set of
coefficients of $P$ is bounded above by $D(m)$ times the supremum norm of $P$
on the unit polydisc. In(\cite{annals}, 2011) it was proved that
\begin{equation}
D(m)\leq\left(  1+\frac{1}{m}\right)  ^{m-1}\sqrt{m}\left(  \sqrt{2}\right)
^{m-1} \label{defant55}%
\end{equation}
which yields the hypercontractivity of the inequality.

The last few years experienced the rising of several works dedicated to
estimating the Bohnenblust--Hille constants (\cite{annals,jfa,diniz2,munn,
jmaa}) and also unexpected connections with Quantum Information Theory
appeared (see, e.g., \cite{montanaro}). There are in fact four cases to be
investigated: polynomial (real and complex cases) and multilinear (real and
complex cases). We can summarize in a sentence the main information from the
recent preprints: the Bohnenblust--Hille constants are, in general,
extraordinarily smaller than the first estimates predicted.

For example, now it is known that the Bohnenblust--Hille constants for the
multilinear case behave in a subpolynomial way. In view of this, the
investigation of lower bounds for the Bohnenblust--Hille constants seems to be
an important task. The existing results for multilinear mappings and real
scalars are highly nontrivial. For instance, in \cite{diniz2} it is shown that
for multilinear mappings and real scalars one has that
\[
C(m)\geq2^{1-\frac{1}{m}}%
\]
and it is still open whether these estimates are sharp or not. Thus, the
possibility of the boundedness of the multilinear Bohnenblust--Hille constants
is open. In this short note we shall show that, if $m\geq2$, then
\[
D(m)\geq\frac{\left(  2+2^{m}\right)  ^{\frac{m+1}{2m}}}{\sqrt{4+2^{m+1}}}.
\]

\section{The result}

To the best of our knowledge, the only nontrivial lower bound for the
constants of the (complex) polynomial Bohnenblust--Hille inequality is%
\[
D(2)\geq1.1066,
\]
proved in \cite{munn}. The lack of known estimates for the norms of complex
polynomials of higher degrees was a barrier to obtain nontrivial lower
estimates for $D\left(  m\right)  $ with $m>2$. Our result provides nontrivial constants:

\begin{theorem}
For every $m\geq2$,%
\[
D(m)\geq\frac{\left(  2+2^{m}\right)  ^{\frac{m+1}{2m}}}{\sqrt{4+2^{m+1}}}>1.
\]
\bigskip
\end{theorem}

\begin{proof}
Let $P_{2}:\ell_{\infty}^{2}\rightarrow\mathbb{C}$ be a $2$-homogeneous
polynomial given by
\[
P_{2}(z_{1},z_{2})=az_{1}^{2}+bz_{2}^{2}+cz_{1}z_{2}.
\]
with $a,b,c\in\mathbb{R}$ and  $ab<0$ and $|c(a+b)|\leq4|ab|.$ From
\cite{munn} we know that
\[
\Vert P_{2}\Vert=\left(  |a|+|b|\right)  \sqrt{1+\frac{c^{2}}{4|ab|}.}%
\]
For each $m$, consider
\[
P_{m}(z)=z_{3}\ldots z_{m}P_{2}(z_{1},z_{2}).
\]

By choosing $a=-b=1$, the $\ell_{\frac{2m}{m+1}}$ norm of the coefficients of
$P_{m}$ is
\[
\left(  \sqrt[m+1]{a^{2m}}+\sqrt[m+1]{b^{2m}}+\sqrt[m+1]{c^{2m}}\right)
^{\frac{m+1}{2m}}=\left(  2+\sqrt[m+1]{c^{2m}}\right)  ^{\frac{m+1}{2m}}%
\]
and
\[
\Vert P_{m}\Vert=\Vert P_{2}\Vert=\left(  \left\vert a\right\vert +\left\vert
b\right\vert \right)  \sqrt{1+\frac{c^{2}}{4\left\vert ab\right\vert }}%
=\sqrt{4+c^{2}}.
\]
Thus, we obtain
\[
D(m)\geq f_{m}(x)
\]
for all real numbers $x$, with
\[
f_{m}(x)=\left(  2+\sqrt[m+1]{x^{2m}}\right)  ^{\frac{\left(  m+1\right)
}{2m}}.\left(  \sqrt{4+x^{2}}\right)  ^{-1}.
\]
By solving $f_{m}^{\prime}(x)=0$ we conclude that $x=2^{\frac{m+1}{2}}$is a
point where the maximum is achieved. Thus
\[
D(m)\geq f_{m}(2^{\frac{m+1}{2}})=\left(  2+2^{m}\right)  ^{\frac{m+1}{2m}%
}.\left(  \sqrt{4+2^{m+1}}\right)  ^{-1}.
\]
A straightforward calculation shows that the last expression is always
strictly greater than $1$.
\end{proof}


\begin{thebibliography}{9}                                                                                                %
\bibitem {bh}H.F. Bohnenblust and E. Hille, On the absolute convergence of
Dirichlet series, Ann. of Math. (2) \textbf{32} (1931), 600--622.

\bibitem {annals}A. Defant, L. Frerick, J. Ortega-Cerd\'{a}, M. Ouna\"{\i}es
and K. Seip, The polynomial Bohnenblust--Hille inequality is hypercontractive,
Ann. of Math. (2) \textbf{174} (2011), 485--497.

\bibitem {jfa}D. Diniz, G.A. Mu\~{n}oz-Fern\'{a}ndez, D. Pellegrino and J.B.
Seoane-Sep\'{u}lveda, The asymptotic growth of the constants in the
Bohnenblust--Hille inequality is optimal, J. Funct. Anal. \textbf{263} (2012), 415--428.

\bibitem {diniz2}D. Diniz, G.A. Mu\~{n}oz-Fern\'{a}ndez, D. Pellegrino and
J.B. Seoane-Sep\'{u}lveda, Lower bounds for the constants in the
Bohnenblust--Hille inequality: the case of real scalars, Proc. Amer. Math.
Soc., in press.

\bibitem {montanaro}A. Montanaro, Some applications of hypercontractive
inequalities in quantum information theory, arXiv:1208.0161v2.

\bibitem {munn}G.A. Mu\~{n}oz-Fern\'{a}ndez, D. Pellegrino, J. Ramos Campos
and J.B. Seoane-Sep\'{u}lveda, A geometric technique to generate lower
estimates for the constants in the Bohnenblust--Hille inequalities, arXiv:1203.0793.

\bibitem {jmaa}D. Pellegrino and J.B. Seoane-Sep\'{u}lveda, New upper bounds
for the constants in the Bohnenblust Hille inequality, J. Math. Anal. Appl.
\textbf{386} (2012), 300--307.
\end{thebibliography}
\end{document}